\documentclass[11pt]{article}

\usepackage{graphicx}
    \graphicspath{ {figures/} }
    \usepackage[margin=1cm]{caption}
    \usepackage{float}
\usepackage{amsmath}
\usepackage{amsthm}
\usepackage{amssymb}
\usepackage[mathscr]{eucal}
\usepackage[all]{xy}
\usepackage{epsfig}
\usepackage{epstopdf}
\usepackage{amscd}
\usepackage[normalem]{ulem}
\usepackage{enumitem}
    \setlist[enumerate,1]{label=\textnormal{(\alph*)}}
\usepackage{thmtools, thm-restate}

\usepackage[left=1.5in,right=1.5in,top=1.5in,bottom=1.5in,
            ]{geometry}
\usepackage{hyperref}

\theoremstyle{plain}
\newtheorem{theorem}{Theorem}

\newtheorem{lemma}[theorem]{Lemma}
\newtheorem{corollary}[theorem]{Corollary}

\newtheorem{conjecture}[theorem]{Conjecture}

\theoremstyle{remark}

\theoremstyle{definition}
\newtheorem*{definition}{Definition} 

\newtheorem*{question*}{Question}

\title{Tile Number and Space-Efficient Knot Mosaics}
\author{Aaron Heap and Douglas Knowles}

\begin{document}

\maketitle
\begin{abstract} In this paper we introduce the concept of a space-efficient knot mosaic. That is, we seek to determine how to create knot mosaics using the least number of non-blank tiles necessary to depict the knot. This least number is called the tile number of the knot. We determine strict bounds for the tile number of a knot in terms of the mosaic number of the knot. In particular, if $t$ is the tile number of a prime knot with mosaic number $m$, then $5m-8 \leq t \leq m^2-4$ if $m$ is even and $5m-8 \leq t \leq m^2-8$ if $m$ is odd.  We also determine the tile number of several knots and provide space-efficient knot mosaics for each of them.
\end{abstract}


\section{Introduction}

Mosaic knot theory is a branch of knot theory that was first introduced by Kauffman and Lomonaco in the paper \textit{Quantum Knots and Mosaics}
\cite{Lom-Kauff} and was later proven to be equivalent to tame knot theory by Kuriya and Shehab in the paper \textit{The Lomonaco-Kauffman Conjecture} \cite{Kuriya}. This approach involves creating a knot mosaic by sectioning off a standard knot diagram into an $n \times n$ array of \textit{mosaic tiles} selected from the collection of eleven tiles shown in Figure \ref{tiles}. Each arc and crossing of the original knot projection is represented by arcs, line segments, or crossings drawn on each tile. These tiles are identified, respectively, as $T_0$, $T_1$, $T_2$, $\ldots$, $T_{10}$. Tile $T_0$ is a blank tile, and we refer to the rest collectively as non-blank tiles.

\begin{figure}[h]
  \centering
  \includegraphics{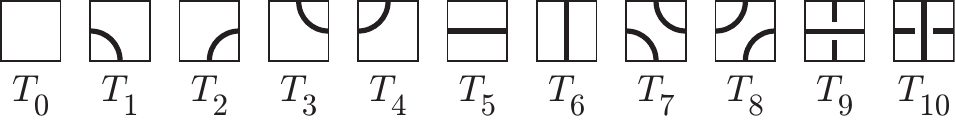}\\
  \caption{Tiles $T_0 - T_{10}$.}
  \label{tiles}
\end{figure}

 A \textit{connection point} of a tile is a midpoint of a tile edge that is also the endpoint of a curve drawn on the tile. A tile is \textit{suitably connected} if each of its connection points touches a connection point of an adjacent tile. Two tiles are \textit{diagonally adjacent} if their array position differs by exactly one row and one column.

\begin{definition}
An $n \times n$ array of tiles is an \textit{$n \times n$ knot mosaic}, or \textit{$n$-mosaic} if each of its tiles are suitably connected.
\end{definition}

Note that an $n$-mosaic could represent a knot or a link, as illustrated in Figure \ref{mosaic-example}. The first two mosaics depicted are $4$-mosaics, and the third one is a $5$-mosaic.

\begin{figure}[h]
  \centering
  \includegraphics{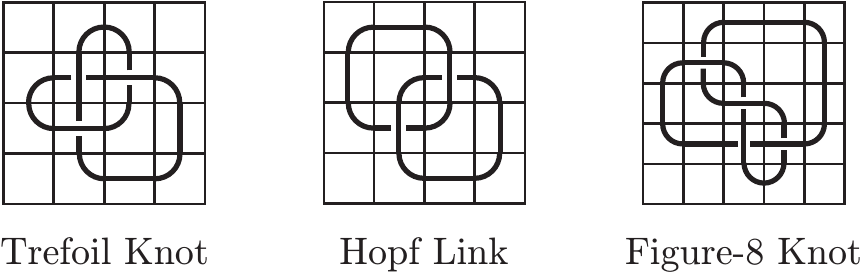}\\
  \caption{Examples of knot mosaics.}
  \label{mosaic-example}
\end{figure}

One particular piece of information of interest is the invariant known as the mosaic number of a knot or link. The \textit{mosaic number} of a knot or link $K$ is the smallest integer $n$ for which $K$ can be represented as an $n$-mosaic. We denote the mosaic number of $K$ as $m(K)$.

Finding bounds on the mosaic number in terms of the crossing number of the knot or link has been a primary focus of research in mosaic knot theory. Lee, Hong, Lee, and Oh, in their paper \textit{Mosaic Number of Knots} \cite{Lee}, found an upper bound for the mosaic number $m$ for a knot or link with crossing number $c$. In particular, for any nontrivial knots and non-split links other than the Hopf link, $m \leq c+1$. In the case of a prime, non-alternating link (except the $6^3_3$ link), they show that $m \leq c-1$.

The mosaic number has previously been determined for every prime knot with crossing number 8 or less. For details, see \textit{Knot Mosaic Tabulations} \cite{Lee2} by Lee, Ludwig, Paat, and Peiffer. In particular, it is known that the mosaic number of the unknot is 2, the mosaic number of the trefoil knot is 4, and the mosaic number of the figure-8 knot (among others) is 5. Every prime knot with eight crossings or less has mosaic number at most 6. In \cite{Heap2}, the authors determine all prime knots that have mosaic number at most 6.

As we work with knot mosaic diagrams, we can move parts of the knot around within the mosaic via \textit{mosaic planar isotopy moves} to obtain another knot mosaic diagram that does not change the knot type of the depicted knot. These are analogous to the planar isotopy moves used to deform standard knot diagrams. A complete list of all of these moves are given and discussed in \cite{Lom-Kauff} and \cite{Kuriya}.

We also point out that throughout this paper we make use of the software package KnotScape \cite{Thistle}, created by Thistlethwaite and Hoste, to verify that a given knot mosaic represents a specific knot.

\section{Space-Efficient Knot Mosaics}

Any given knot can be represented as a knot mosaic in many different ways, sometimes represented on a mosaic that is larger than necessary or inefficiently represented with unnecessary features or empty space within the diagram causing the mosaic to have more non-blank tiles than absolutely necessary. In this paper, we want to explore, in some sense, the most efficient way to represent a knot as a knot mosaic. A few examples of the trefoil knot are given in Figure \ref{ineff-mosaics}. It seems clear that the middle two knot mosaics are not represented in an overly efficient way. The first and last knot mosaics in Figure \ref{ineff-mosaics} are similar to each other, but the first uses thirteen non-blank tiles and the last uses only twelve non-blank tiles. Of the four mosaics depicted, the last one uses the least amount of space within the smallest possible mosaic.

\begin{figure}[h]
  \centering
  \includegraphics{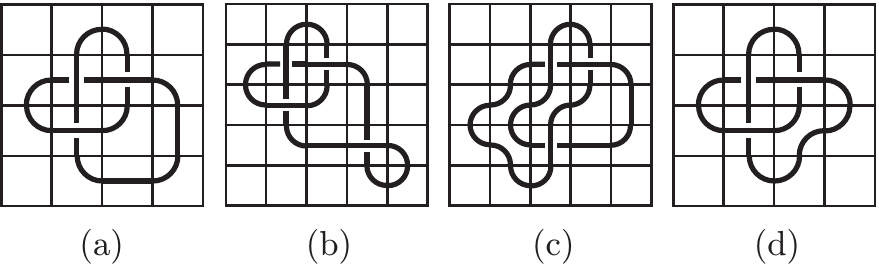}\\
  \caption{Examples of trefoil knot mosaics.}
  \label{ineff-mosaics}
\end{figure}

\begin{definition}
A knot mosaic is called \textit{minimal} if it is a realization of the mosaic number of the knot. That is, a knot with mosaic number $n$ is depicted as an $n$-mosaic.
\end{definition}

\begin{definition}
A knot mosaic is called \textit{reduced} if there are no unnecessary, reducible crossings in the knot mosaic diagram. That is, we cannot draw a simple, closed curve on the knot mosaic that intersects the knot diagram transversely at a single crossing but does not intersect the knot diagram at any other point.
\end{definition}

\begin{definition}
The \textit{tile number of a mosaic} is the number of non-blank tiles (all tiles except $T_0$) used to create that specific mosaic.
\end{definition}

\begin{definition}
The \textit{tile number $t(K)$ of a knot or link} $K$ is the fewest non-blank tiles needed to construct $K$. That is, it is the smallest possible tile number of all possible mosaic diagrams for $K$.
\end{definition}

\begin{definition}
The \textit{minimal mosaic tile number $t_M(K)$ of a knot or link} $K$ is the fewest non-blank tiles needed to construct $K$ on a minimal mosaic. That is, it is the smallest possible tile number of all possible minimal mosaic diagrams for $K$.
\end{definition}

It is known that the crossing number of a knot cannot always be realized on a minimal mosaic, such as the $6_1$ knot. Ludwig, Evans, and Paat \cite{Ludwig} constructed an infinite family of such knots. Similarly, the authors show in \cite{Heap2} that the tile number of a knot cannot always be realized on a minimal mosaic, using the $9_{10}$ knot as an example with mosaic number 6 and $t_M(9_{10})=32$, but $t(9_{10})=27$ is realized on a 7-mosaic.

\begin{definition}
A knot $n$-mosaic is \textsl{space-efficient} if it is reduced and if the tile number has been minimized on an $n \times n$ mosaic through a sequence of planar isotopy moves when the size of the mosaic remains unchanged.
\end{definition}

\begin{definition}
A knot mosaic is \textsl{minimally space-efficient} if it is minimal and space-efficient.
\end{definition}

The knot mosaic depicted in Figure \ref{ineff-mosaics}(a) is minimal but not space-efficient because we can take the arc that passes through the bottom, right corner tile of the mosaic and push it into the diagonally adjacent tile location, thus decreasing the number of non-blank tiles used in the mosaic, and the result is the knot mosaic depicted in Figure \ref{ineff-mosaics}(d), which is minimally space-efficient.

On a minimally space-efficient knot mosaic, the minimal mosaic tile number of the depicted knot must be realized, but the tile number of the knot might not be realized. There may be a larger, non-minimal knot mosaic that uses fewer non-blank tiles, meaning that a space-efficient knot mosaic need not be minimally space-efficient.

\section{Bounds for the Tile Number}\label{counting-tools}

As we seek to determine the tile numbers of knots and find minimally space-efficient knot mosaics for them, we will be working with a large number of possible placements of tiles on a mosaic. To help us simplify explanations and figures, we adopt a few conventions. In particular, we will make use of \textit{nondeterministic tiles} when there are multiple options for the tiles that can be placed in specific tile locations of a mosaic. We will usually denote these as dashed arcs or line segments on the tile. Some examples of these are shown in the first five tiles of Figure \ref{nondeterministic}.

\begin{figure}[h]
  \centering
  \includegraphics{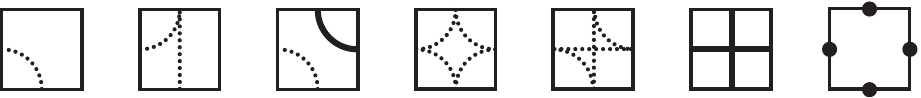}\\
  \caption{Examples of nondeterministic tiles.}
  \label{nondeterministic}
\end{figure}

The first tile in Figure \ref{nondeterministic} could be a single arc tile or a blank tile. The second one could be a single arc tile or a line segment tile. The third tile could be a single arc tile or a double arc tile, but the depicted solid arc is necessary. The fourth one could be any single arc or double arc tile. The fifth one could be a double arc tile or a crossing tile. The sixth tile shown in Figure \ref{nondeterministic} must be a crossing tile, but the crossing type is not yet determined. A point on the edge of a tile indicates a required connection point for the tile. The last tile in Figure \ref{nondeterministic} must have four connection points, and that tile must be either a double arc tile or a crossing tile.

If there is a connection point at the top or bottom of a tile, we may say that there is a connection point \emph{entering} the row that contains that tile. Similarly, a connection point is entering a column if there is a connection point at the right or left of a tile in that column. The connection point may be referred to as an \emph{entry point} for the row or column. For example, if there is a connection point between a tile in the third row and a tile in the fourth row of a mosaic, then that connection point is an entry point point for the third row and an entry point for the fourth row. If a row or column of a mosaic has at least one non-blank tile in it, we may say that the row or column is \emph{occupied}.

The tiles in the outer most rows and columns are referred to as \emph{boundary tiles}. The \textit{inner board} of an $n \times n$ mosaic is the $(n-2) \times (n-2)$ array of tiles that remain after removing the boundary tiles. The first and last boundary tiles in the first and last row of the mosaic are called \emph{corner tiles}. Suppose there are two adjacent single arc tiles that share a connection point, and the other connection points enter the same adjacent row or column. The four options are shown in Figure \ref{caps}, and we will refer to these collectively as \emph{caps} and individually as \emph{top caps}, \emph{right caps}, \emph{bottom caps}, and \emph{left caps}, respectively.
\begin{figure}[h]
  \centering
  \includegraphics{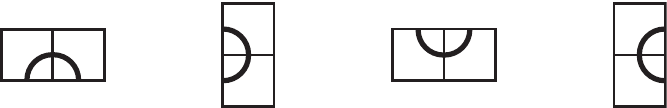}\\
  \caption{A top cap, right cap, bottom cap, and left cap, respectively.}
  \label{caps}
\end{figure}

Equipped with this terminology, we consider the following lemmas that will assist us in counting the minimum number of non-blank tiles necessary to create knot mosaics. We point out that some of these apply to mosaics of any knots and links, while others only apply to mosaics of prime knots. This first lemma tells us that we can create all of our space-efficient knot mosaics without using the corner tile locations. Because the the outer rows and columns need not be occupied, we may assume that the first tile and the last tile in the first occupied row and column is a blank tile, and similarly for the last occupied row and column.

\begin{lemma}\label{no-corners}
Suppose we have a space-efficient $n$-mosaic with $n \geq 4$ and no unknotted, unlinked link components. Then the four corner tiles are blank $T_0$ tiles (or can be made blank via a planar isotopy move that does not change the tile number). The same result holds for the first and last tile location of the first and last occupied row and column.
\end{lemma}

\begin{proof}
We prove that the top, left corner must be blank, and the proof for the remaining three corners of the mosaic is similar. If the top, left tile is not blank, it must be the single arc $T_2$ tile. We simply run through the eleven possible mosaic tiles that could be placed in diagonally adjacent tile position. Each option is depicted in Figure \ref{no-corners1}. In each case except the first one, the mosaic either has a trivial unlinked component or is not space-efficient. In the first case, we can push the arc tile in the corner position into the diagonally adjacent tile position without changing the tile number.

\begin{figure}[h]
  \centering
  \includegraphics{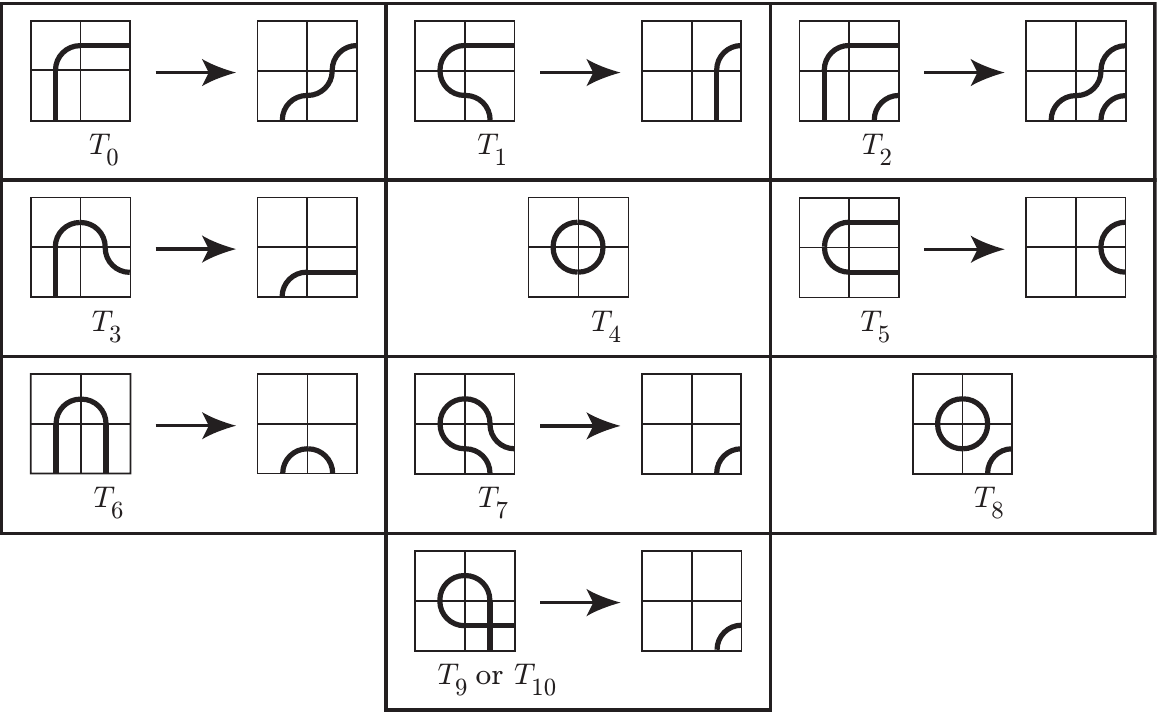}\\
  \caption{All possible upper, left $2 \times 2$ sub-arrays if the upper, left corner is not blank.}
  \label{no-corners1}
\end{figure}

We note that none of this argument hinges on the fact that the assumed blank tile must be in the first row and first column. It only requires that the rows above it and the columns to the left of it are blank. Thus, the result applies not only to the tile in the first row and first column, but also to the tile in the first occupied row and column.
\end{proof}

\begin{lemma}\label{even-connections} For any knot mosaic, if a row (or column) is occupied, then there are at least two non-blank tiles in that row (or column). In fact, there are an even number of entry points between any two rows (or columns).
\end{lemma}

\begin{proof} This lemma should be quite obvious, as knots and link components are simple closed curves. If there is an entry point from Row A into Row B, then a strand of the knot or a link component has entered Row B. In order to connect back to the rest of the knot or link component and complete the circle, that strand must pass back into Row A at some other entry point, necessarily on some other tile, and these entry points must come in pairs. The same is true for columns.
\end{proof}

\begin{lemma}\label{four-below} Suppose we have a space-efficient $n$-mosaic with $n \geq 4$ and no unknotted, unlinked link components. If there is a cap in any row (or column), then the two tiles that share connection points with the cap must have four connection points. The same result holds if the arc tiles in the cap are not adjacent but have one or more line segment tiles between them.
\end{lemma}

\begin{figure}[h]
  \centering
  \includegraphics{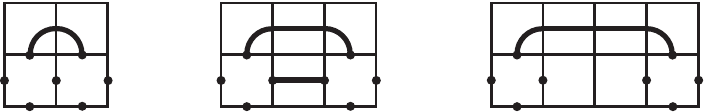}\\
  \caption{If there is a strand of a knot or link in a row (or column) with both entry points coming from the same row (or column), then the tiles that share these entry points must have four connection points.}
  \label{four-below-1}
\end{figure}

\begin{proof}  We focus on rows, and the result for columns follows via a rotation of the mosaic. Suppose we have a top cap, as in the first diagram of Figure \ref{four-below-1}. We need to prove that the two tiles just below it must both have four connection points. Assume at least one of these tiles only has two connection points. Then each of the possibilities, except those resulting in a trivial unlinked link component, are shown in Figure \ref{four-below-2}, and they are not space-efficient since the tile number can be decreased. So both tiles connecting to the arcs must have four connection points. The cases involving the other types of caps are covered by a rotation of this.

\begin{figure}[h]
  \centering
  \includegraphics{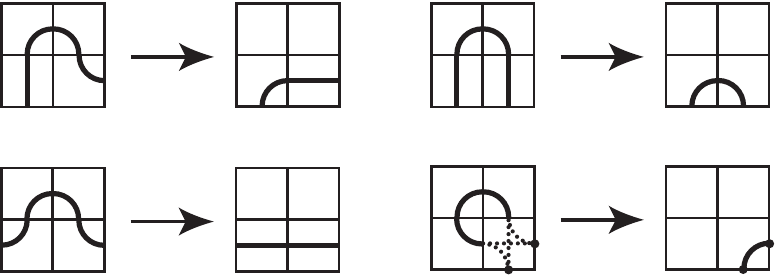}\\
  \caption{If tiles in the second row do not both have four connection points, the mosaics are not space-efficient or have an unnecessary loop.}
  \label{four-below-2}
\end{figure}

Now suppose the two arcs in the cap are not adjacent but connected by a horizontal line segment, as in the second diagram in Figure \ref{four-below-1}.  Again we need to prove that the two tiles below the arc tiles must both have four connection points. Assume at least one of these tiles only has two connection points. Each of the possibilities, except those resulting in a trivial unlinked link component, are shown in Figure \ref{four-below-3}, and none of them are space-efficient. So both tiles connecting to the arcs must have four connection points.
\begin{figure}[h]
  \centering
  \includegraphics{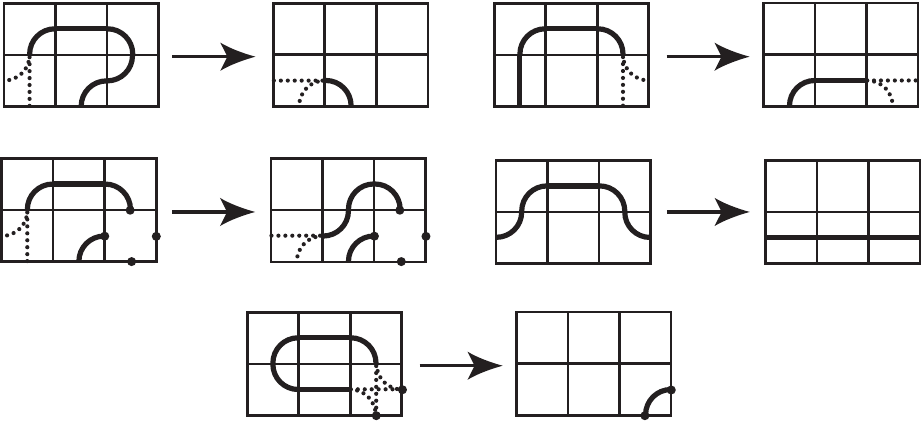}\\
  \caption{If tiles in the second row do not both have four connection points, they are not space-efficient.}
  \label{four-below-3}
\end{figure}
The cases are similar if there is more than one line segment tile connecting the two arc tiles, and again, the cases involving the other caps are covered by a rotation of this.
\end{proof}

\begin{lemma}\label{two-arcs1} Suppose we have a space-efficient mosaic of a prime knot. If there is an occupied row (or column) with less than four non-blank tiles, then the mosaic can be simplified so that the row (or column) has exactly two non-blank tiles in the form of a cap.
\end{lemma}

\begin{proof} If a row is occupied and there are less than four non-blank tiles in the row, then there are either two or three non-blank tiles in the row by Lemma \ref{even-connections}. By the same lemma, there must be an even number of entry points at the top of the row and at the bottom of the row. As there are no more than three non-blank tiles, this means there are either zero or two entry points at the top of the row and zero or two entry points at the bottom of the row. If all of the non-blank tiles are vertical segment tiles, this row can be collapsed by shifting the rows below it upward. All other possibilities, up to rotation or reflection, are shown in Figure \ref{two-arcs}.

\begin{figure}[h]
  \centering
  \includegraphics{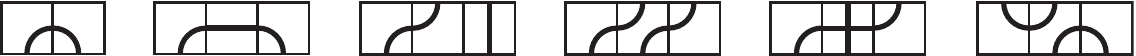}\\
  \caption{Only possibilities to have less than four non-blank tiles in a single row.}
  \label{two-arcs}
\end{figure}

The last possibility results in at least two (unlinked) link components. Because the mosaic must depict a prime knot, each of the third, fourth, and fifth possibilities in Figure \ref{two-arcs} is not space-efficient, as the portion of the knot either above or below this row must be unknotted and would simplify to one of the first two possibilities. Also, the knot mosaic with the fifth possibility is not reduced, as the crossing can be removed by a flip.

Consider the second option in Figure \ref{two-arcs}, with a horizontal line segment between the two single are tiles. We claim that the mosaic is either not space-efficient or the horizontal segment can be collapsed in a way that does not change the tile number. Because there are no other non-blank tiles in this row and the knot mosaic does not depict a link, we know all tiles above this row must be blank. Lemma \ref{four-below} tells us that this row and the row below it must be as in the second picture of Figure \ref{four-below-1}, with two horizontal segment tiles in the same column. If all of the tiles in this column below the horizontal segments are blank or horizontal line segment tiles, then the mosaic is not space-efficient as the tile number can be decreased by collapsing this column. Consider the first tile in this column that is not blank or a horizontal segment. Because the tile above it is blank or a horizontal line segment, this tile can only be a single arc tile $T_1$ or $T_2$. In either case, the horizontal segment tiles can be collapsed without changing the tile number or the mosaic is not space-efficient, as the tile number can be decreased by collapsing the horizontal segments. If the first non-blank, non-horizontal tile is the $T_1$ arc tile, then the collapse of the horizontal segments is done as in one of the options in Figure \ref{two-arcs-collapse}, possibly with more blank or horizontal segment tiles above the arc tile. Rotations and reflections of these cover all other cases.

\begin{figure}[h]
  \centering
  \includegraphics{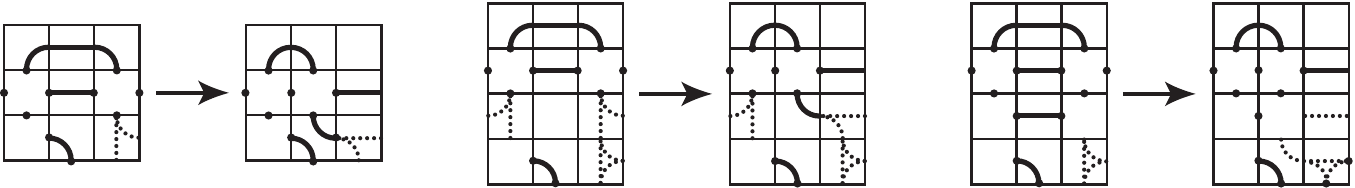}\\
  \caption{Collapsibility of horizontal segments.}
  \label{two-arcs-collapse}
\end{figure}

Therefore, since the mosaic is space-efficient, we can always alter the given mosaic via planar isotopy moves so that any row or column with less than four non-blank tiles has exactly two non-blank tiles in the form of a cap, and this alteration does not change the tile number.
\end{proof}

Because of this lemma, in a space-efficient mosaic of a prime knot, we may assume that every occupied row and column has at least four non-blank tiles or can be simplified to a single cap.

\begin{corollary} Every space-efficient mosaic of a prime knot can be drawn so that every row and column has either 0, 2, 4, or more non-blank tiles.
\end{corollary}

\begin{lemma}\label{top-caps} Suppose we have a space-efficient $n$-mosaic of a knot or link. Then the first occupied row of the mosaic can be simplified so that the non-blank tiles form only top caps. In fact, there will be $k$ top caps for some $k$ such that $1 \leq k \leq (n-2)/2$. Similarly, the last occupied row is made up of bottom caps, and the first and last occupied columns are made up of left caps and right caps, respectively.
\end{lemma}

\begin{proof} Because we are considering the first occupied row of the mosaic, there can be no connection points along the top of the row. So the row must consist entirely of blank tiles, top caps, or $T_1$ and $T_2$ single arc tiles separated by any number of horizontal segment tiles. If there is only one horizontal segment tile between the arc tiles, this can be reduced to a top cap without changing the tile number via the same argument in the proof of Lemma \ref{two-arcs1}. If there are two horizontal segment tiles between the arc tiles, then the we can eliminate them via a planar isotopy move without changing the tile number. Because of space-efficiency, eventually there must be single arc tiles below the two horizontal segments. We have shown in Figure \ref{top-caps-1} the planar isotopy moves for the cases where this occurs in the second or third occupied row. If it happens in a later row, the moves are similar.

\begin{figure}[h]
  \centering
  \includegraphics{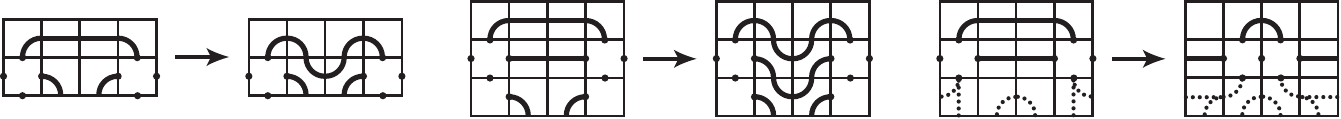}\\
  \caption{Simplifying to top caps.}
  \label{top-caps-1}
\end{figure}

If there are more than two horizontal segment tiles between the arc tiles in the first occupied row, we can eliminate consecutive pairs as above, reducing the number of horizontal segment tiles between two arc tiles to one or none, and we can eliminate the single horizontal line segments as we did in the proof of Theorem \ref{two-arcs1}. In any case, we are able to reduce everything to a collection of top caps.

Because there are $n$ tile locations in the first occupied row and, by Lemma \ref{no-corners}, we can assume the first and last tiles in this row are blank, there are only $n-2$ tiles to place the top caps. Therefore, there are at most $(n-2)/2$ top caps. Rotations of this prove the result for the first and last occupied columns and rows.
\end{proof}

\begin{lemma}\label{mid-rows} Suppose we have a space-efficient mosaic of a prime knot with at least five occupied rows. Then every occupied row except the first two and last two occupied rows has at least five non-blank tiles.
\end{lemma}

\begin{proof} We begin with a space-efficient mosaic of a prime knot with at least five occupied rows. By Lemma \ref{top-caps}, we know that the first occupied row is made up of top caps. To avoid space-inefficiency, composite knots, and multi-component links, we know there must be at least four connection points between any two rows, except possibly between the first and second occupied rows and between the last two occupied rows. This means that in any one of these intermediate rows, there must be at least four connection points along the top of the row and at least four connection points along the bottom of the row. If there are more than four in any given row, then there are more than four non-blank tiles in that row. Suppose there are exactly four connection points along the top and along the bottom of one of these intermediate rows. If the four connection points at the top of this row are vertically aligned with the four connection points at the bottom of the row, then these four non-blank tiles must all be vertical segment tiles, and the resulting mosaic would not be space-efficient. Thus, they are not vertically aligned, and there are at least five non-blank tiles in this row. Therefore, other than the first two occupied rows and the last two occupied rows, every row must have at least five non-blank tiles.
\end{proof}

These lemmas combine to provide bounds for the tile number. We have an upper bound for the tile number of a general $n$-mosaic of any knot or link, and we have a lower bound for an $n$-mosaic of any prime knot.

\begin{theorem}\label{bounds-mosaic} For $n \geq 4$, suppose we have a space-efficient $n$-mosaic of a knot or link $K$ with no unknotted, unlinked link components, and either every row or every column of the mosaic is occupied. If $n$ is even, then the tile number of the mosaic is less than or equal to $n^2-4$. If $n$ is odd, then the tile number of the mosaic is less than or equal to $n^2-8$. If $K$ is a prime knot, then the tile number is greater than $5n-8$.
\end{theorem}

\begin{proof} Suppose we have a space-efficient $n$-mosaic of $K$ in which either every row or every column is occupied. By Lemma \ref{no-corners}, we know we do not need to use the corners of the mosaic. In the case where $n$ is odd, Lemma \ref{even-connections} forces one more blank tile in each outer row and column because we can only have an even number of non-blank tiles in each of these. Therefore, the tile number of this mosaic must be less than or equal to either $n^2-4$ or $n^2-8$, depending on whether $n$ is even or odd.

Now suppose that $K$ is a prime knot, and assume every row of the mosaic is occupied. By Lemma \ref{two-arcs1}, we may assume that the first row of the mosaic either has at least four non-blank tiles or has exactly two non-blank tiles in it, a top cap. Assuming the latter, Lemma \ref{four-below} tells us that the next row down has at least four non-blank tiles. Lemma \ref{mid-rows} tells us that the rest of the rows must have at least five non-blank tiles, except possibly the last and next to last rows. At a minimum, since all rows are occupied, the last row must have at least two non-blank tiles (a bottom cap), and the next to last row has at least four non-blank tiles. Thus there are at least two non-blank tiles in the first and last rows, at least four non-blank tiles in the second and next to last rows, and at least five non-blank tiles in each of the $n-4$ intermediate rows, providing a minimum of $5n-8$ non-blank tiles in the $n$-mosaic. A rotation of this gives the same result if every column is occupied.
\end{proof}

\begin{corollary}\label{bounds-knot} Suppose we have a knot or link $K$ with mosaic number $m(K)=m$ for $m \geq 4$ and no unknotted, unlinked link components. If $m$ is even, then $t(K) \leq m^2-4$. If $m$ is odd, then $t(K) \leq m^2-8$. If $K$ is a prime knot, then $t(K) \geq 5m-8$.
\end{corollary}

\begin{proof} Let $K$ be a knot or link with mosaic number $m$. Then we know $K$ can be drawn on an $m$-mosaic (or larger) but not a smaller mosaic, and the bounds for $t(K)$ follow immediately from the theorem, even when the tile number of a prime knot does not occur on a minimal mosaic.
\end{proof}

Of course the restrictions on the knot or link given in Corollary \ref{bounds-knot} naturally lead to a question about $n$-mosaics of composite knots and multi-component links.

\begin{question*} What are the bounds for the tile numbers of links and composite knots with mosaic number $m$?
\end{question*}

\section{Tile Numbers of Small Knot Mosaics}

Let us first consider the smallest mosaics, that is, $n$-mosaics with $n\leq 5$. We begin with the unknot, which has mosaic number 2.

\begin{theorem} The tile number (and minimal mosaic tile number) of the unknot is $t(\textnormal{unknot})=4$.
\end{theorem}

\begin{proof}  The least number of non-blank tiles necessary to create the unknot is four, and this is shown on a minimal mosaic in the first mosaic of Figure \ref{unknot}.
\end{proof}
\begin{figure}[h]
  \centering
  \includegraphics{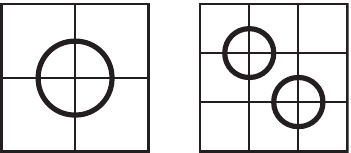}\\
  \caption{The unknot and two component unknotted unlink.}
  \label{unknot}
\end{figure}

Kauffman and Lomonaco \cite{Lom-Kauff} show that the only knots or links that fit on a 3-mosaic are the unknot or the two component unknotted unlink. The latter of which has tile number and minimal mosaic tile number 7, as seen in Figure \ref{unknot}. All other knots and links have mosaic number 4 or more.

\begin{theorem} The tile number (and minimal mosaic tile number) of any knot or nontrivial link with mosaic number 4 is 12. For links with unlinked components and mosaic number 4, the possible tile numbers (and minimal mosaic tile numbers) are 10, 13, and 16.
\end{theorem}

\begin{proof}
For prime knots, this is a direct result of Corollary \ref{bounds-knot}, which says that when the mosaic number is 4, the tile number is bounded above and below by 12. For composite knots and nontrivial links with mosaic number 4, Corollary \ref{bounds-knot} only says that the upper bound is 12. As long as the link is nontrivial, there must be at least two crossing tiles in the mosaic. To be space-efficient, any suitably connected knot mosaic with at least two crossing tiles must have tile number at least 12, and up to symmetry, the only options are shown in Figure \ref{4-mosaics-l}.

\begin{figure}[h]
  \centering
  \includegraphics{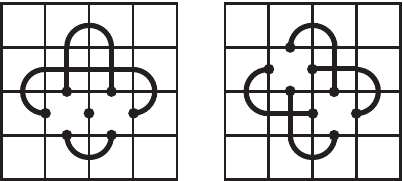}\\
  \caption{Possible nontrivial space-efficient 4-mosaics.}
  \label{4-mosaics-l}
\end{figure}

Kuriya and Shehab \cite{Kuriya} find a complete list of all possible 4-mosaics. The mosaics found there are not necessarily space-efficient, but it is not difficult to make them space-efficient. The results are shown in Figure \ref{4-mosaics}.
\end{proof}

\begin{figure}[h]
  \centering
  \includegraphics{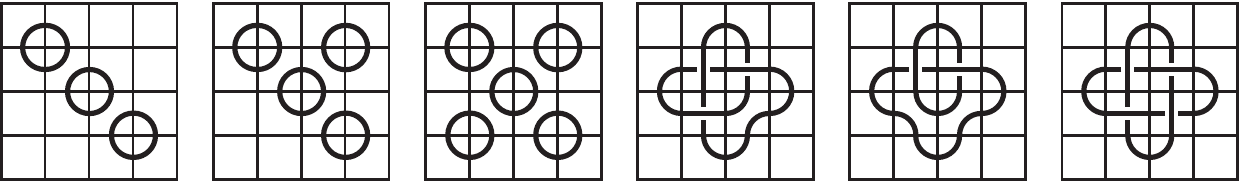}\\
  \caption{Space-efficient knot mosaics for knots and links with mosaic number 4.}
  \label{4-mosaics}
\end{figure}

\begin{corollary} The tile number (and minimal mosaic tile number) of the trefoil knot, Hopf link, and Solomon's knot is $t(3_1)=t(2^2_1)=t(4^2_1)=12$.
\end{corollary}

We now seek to find the possible tile numbers of space-efficient $5$-mosaics, and find the tile number of all knots and links with mosaic number 5. For a prime knot with mosaic number 5, Corollary \ref{bounds-knot} tells us that the tile number is bounded above and below by 17. For a composite knot or link $K$ with mosaic number 5, Corollary \ref{bounds-knot} provides an upper bound $t(K) \leq 17$. Just a little more work is required to show that this is also the lower bound.

\begin{theorem} The tile number (and minimal mosaic tile number) of any knot or link $K$ with mosaic number 5 and no unknotted, unlinked components is $t(K)=17$. This includes the prime knots $4_1$, $5_1$, $5_2$, $6_1$, $6_2$, and $7_4$. Moreover, any space-efficient 5-mosaics of $K$ has a layout as shown in Figure \ref{5-mosaics-l}.
\begin{figure}[h]
  \centering
  \includegraphics{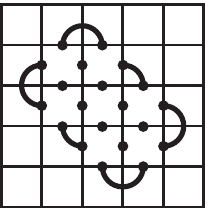}
  \caption{Only possible layout for a space-efficient 5-mosaic.}
  \label{5-mosaics-l}
\end{figure}
\end{theorem}

\begin{proof}
Because the mosaic number of $K$ is 5, either every row or every column must be occupied. Assume every row of the mosaic is occupied. By Lemma \ref{top-caps}, we may assume that the first row of the $5$-mosaic has two non-blank tiles, a top cap. By Lemma \ref{four-below}, the second row must have at least four non-blank tiles. Similarly, the last row has two non-blank tiles, and the next to last row has at least four non-blank tiles. Now we observe the middle row. There are at least four entry points at the top of this row and four entry points at the bottom of it. If there are exactly four non-blank tiles in this row, then this means that the entry points at the top of the row are vertically aligned with the entry points at the bottom of the row, and the four non-blank tiles in this row must be vertical line segments, which means that the mosaic is not space-efficient. Therefore, there must be five non-blank tiles in the middle row, giving us a minimum tile number of 17.

Minimally space-efficient mosaics for the knots and links mentioned are provided in Figure \ref{5-mosaics}.
\end{proof}

\begin{figure}[h]
  \centering
  \includegraphics{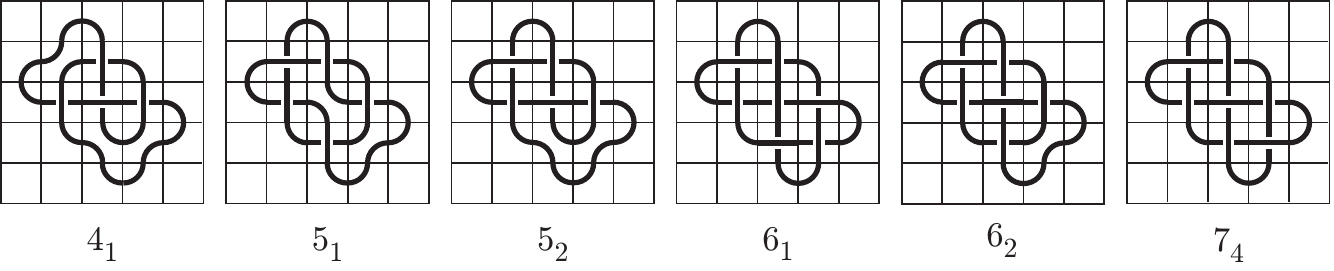}
  \caption{Space-efficient knot mosaic for prime knots with mosaic number 5.}
  \label{5-mosaics}
\end{figure}


\section{Tile Numbers of Knots with Mosaic Number 6}

Now we wish to do for 6-mosaics what we have done for the smaller mosaics. We seek to find the possible tile numbers of all 6-mosaics. In this section, we only consider mosaics of prime knots. Theorem \ref{bounds-mosaic} gives us the bounds for the tile number of any space-efficient prime knot. In particular, suppose we have a prime knot $K$ on a space-efficient $n$-mosaic. If $n=6$, then the tile number $t$ of the mosaic is $22 \leq t \leq 32$. If $n=7$, then the tile number $t$ of the mosaic is $27 \leq t(K) \leq 41$. This leads us to some immediate corollaries to Theorem \ref{bounds-mosaic}.

\begin{corollary}\label{tile-num=min-mos-tile-num} For any prime knot $K$ with mosaic number $m(K)\leq 6$, if the minimal mosaic tile number $t_M(K) \leq 27$, then the tile number of $K$ equals the minimal mosaic tile number of $K$.
\end{corollary}

\begin{proof} We already knew this result for $m(K) \leq 5$. Since a $7$-mosaic or larger cannot have tile number smaller than 27, we know that for any prime knot with mosaic number 6 and minimal mosaic tile number at most 27, the number of non-blank tiles cannot be decreased by placing it on a larger mosaic.
\end{proof}

We can now determine the tile number of all prime knots with crossing number 7 or less and several prime knots with crossing number 8 or 9.

\begin{corollary}\label{t=22} If $K$ is one the following prime knots, then the tile number of $K$ is $t(K)=22$:
\begin{enumerate}
  \item $6_3$,
  \item $7_1$, $7_2$, $7_3$, $7_5$, $7_6$, $7_7$,
  \item $8_1$, $8_2$, $8_3$, $8_4$, $8_7$, $8_8$, $8_9$, $8_{13}$,
  \item $9_5$, and $9_{20}$.
\end{enumerate}
\end{corollary}

\begin{proof}  We have given minimally space-efficient mosaics with tile number 22 for each of these knots in Figure \ref{6-mosaics}. Since the mosaic number of each of these knots is 6, we know that they cannot have a tile number smaller than 22.
\end{proof}

\begin{figure}[h]
  \centering
  \includegraphics{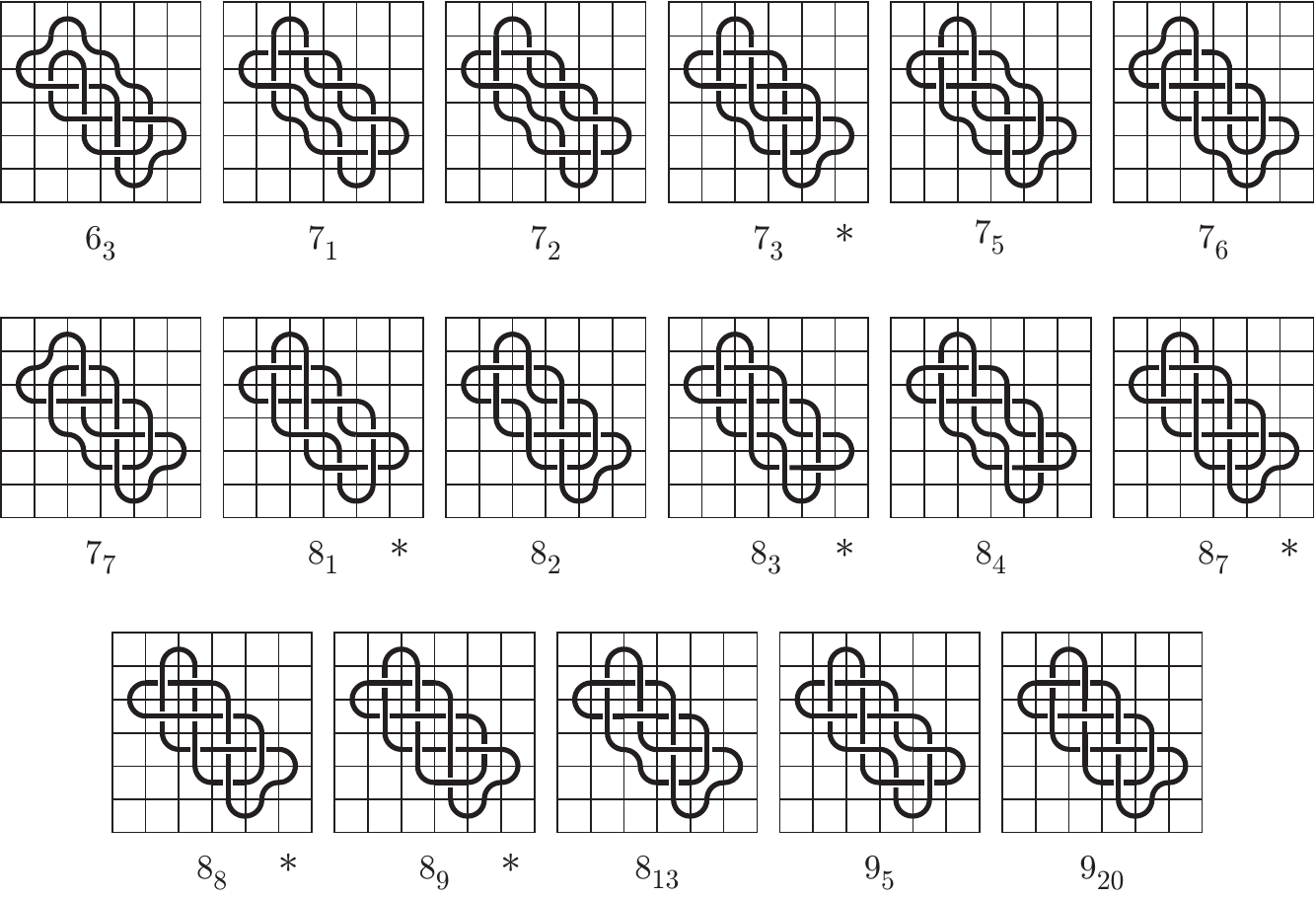}
  \caption{Space-efficient knot mosaics for prime knots known to have mosaic number 6 and tile number 22.}
  \label{6-mosaics}
\end{figure}

For many of these minimally space-efficient knot mosaics in Figure \ref{6-mosaics}, the crossing number of the knot was also realized. However, this is not always possible. In order to obtain the minimally space-efficient knot mosaic for $7_3$, for example, we had to use eight crossings. In \cite{Heap2}, the authors find that none of the possible minimally space-efficient knot mosaics with twenty-two non-blank tiles and exactly seven crossings produced $7_3$. The fewest number of non-blank tiles needed to represent $7_3$ with only seven crossings is twenty-four. Thus, on a minimally space-efficient knot mosaic, for the tile number (or minimal mosaic tile number) to be realized, it might not be possible for the crossing number to be realized. Each such knot mosaic in Figure \ref{6-mosaics} is labeled with an asterisk ($\ast$).

We have established that the tile number of a space-efficient prime knot 6-mosaic is between 22 and 32, but we can be more specific. We have already given examples of knot mosaics with tile number 22, and we now determine the other possible tile numbers.

\begin{theorem}\label{tile-numbers} If we have a space-efficient 6-mosaic of a prime knot $K$ for which either every column or every row is occupied, then the only possible values for the tile number of the mosaic are 22, 24, 27, and 32. Furthermore, any such mosaic of $K$ is equivalent (up to symmetry) to one of the mosaics in Figure \ref{tile-numbers-1a}.
\begin{figure}[h]
  \centering
  \includegraphics{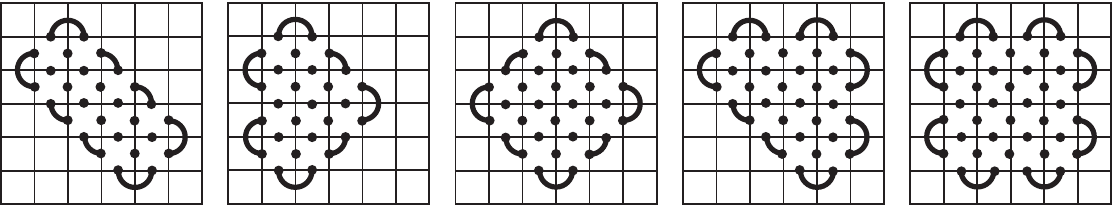}\\
  \caption{Only possible layouts for a space-efficient 6-mosaic.}
  \label{tile-numbers-1a}
\end{figure}
\end{theorem}

\begin{proof}  If there are two non-blank tiles (one top cap) in the first row, we claim the second row must have four non-blank tiles. If it had more than four, there would be at least one horizontal segment tile in this row, and this will cause the mosaic to not be space-efficient. The same result holds for the second occupied row from the bottom and the second occupied columns from the right or left. To prove the claim, we consider the possible locations of the top cap in the first occupied row.

Suppose there is a top cap in the first two tile positions after the corner tile. Then the first tile in the second occupied row must be a single arc tile $T_2$, followed by two tiles with four connection points. If the next two tiles are both horizontal segment tiles, this forces the arc tile $T_1$ into the last position in this row, which is necessarily part of a right cap, and the previous tile position with the horizontal segment should have had four connection points by Lemma \ref{four-below}. If there is only one horizontal segment, then the fifth tile position is the arc tile $T_1$. Assume this is not part of a cap, and look at the tile directly below the horizontal segment. Because there is no connection point at the top of this tile, it can only be a horizontal segment, $T_1$ arc, $T_2$ arc, or blank tile. If it is a horizontal segment tile, then the mosaic is not space-efficient because either everything in this column is a horizontal segment or blank tile or the mosaic is as depicted in Figure \ref{no-segments}(a), in which the entire upper, left $3 \times 3$ corner of the mosaic can be shifted to the right, collapsing the horizontal segments. If it is a $T_1$ tile, the knot is not space-efficient, as we can see in Figure \ref{no-segments}(b), and the tile number can be decreased by pushing the horizontal segment and $T_1$ tiles from the second row into the $T_1$ tile in the third row. If it is a $T_2$ tile or blank tile, the mosaic must be as in Figure \ref{no-segments}(c), and either the depicted knot is not prime or the mosaic is not space-efficient, as shown by the dashed line cutting through the knot.

\begin{figure}[h]
  \centering
  \includegraphics{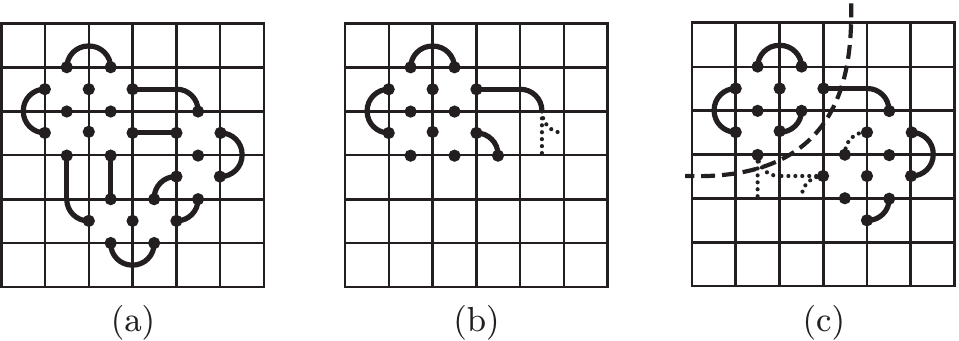}\\
  \caption{Possible configurations with a horizontal segment in the fourth tile position of the second row.}
  \label{no-segments}
\end{figure}

If there is a top cap in the second and third tile positions after the corner tile, it is easy to see that a horizontal segment tile is not allowed in the second row. If there was one, this would force a single arc tile into a boundary column, which is necessary part of a cap, and the tile position with the horizontal segment should have had four connection points by Lemma \ref{four-below}. This completes the proof of our claim that the second row has no horizontal segment tiles.

Suppose every row of the mosaic is occupied. By Lemma \ref{top-caps}, the first row has either one or two top caps, that is, two or four non-blank tiles. We just showed that if there are two non-blank tiles in the first row, the second row has exactly four non-blank tiles. If there are four non-blank tiles in the first row, by Lemma \ref{four-below} we know that the second row will have four tiles with four connection points each, along with one single arc tile $T_1$ and one single arc tile $T_2$, meaning the second row must have six non-blank tiles. By Lemma \ref{mid-rows}, we know that the middle two rows have five or six non-blank tiles. The non-blank tiles in the last two rows are counted as they were in the first two rows. Analogously, we know how many non-blank tiles can be in each of the columns. If not every row of the mosaic is occupied, then every column must be, and a similar argument applies. With all of this in mind, the five layouts depicted in the theorem are the only possible configurations, up to rotation, reflection, or translation, of the non-blank tiles.

Now we turn our attention to the connection points. Notice that all of the nondeterministic tiles must have four connection points. Most of them are there because Lemma \ref{four-below} requires it. In the second layout, for example, all of the connection points are required by Lemma \ref{four-below}. In the remaining four layouts, the only connection points that are not required by Lemma \ref{four-below} are the four connection points on the tile edges that meet at the center point of the mosaic. In the first, fourth, and fifth layouts, if any of these four connection points are missing, then either the knot is not prime or the mosaic is not space-efficient because there would be only two connection points between the third and fourth columns. In the third layout, if any of these four connection points is missing, then there would only be eight tile locations with four connection points. If all eight of these are crossing tiles, the result is a two component link. If less than eight of them are crossings, we know the mosaic is not space-efficient because every prime knot with seven crossings or less has tile number less than 24.
\end{proof}

Now that all of the possible layouts for a space-efficient 6-mosaic have been determined, all that remains for 6-mosaics is to determine what knots have mosaic number 6 and then determine their tile numbers. This task is completed in \cite{Heap2}, where the authors determine minimally space-efficient knot mosaics for all prime knots with mosaic number 6. Consequently, the minimal mosaic tile number (and tile number in the case $t(K)=t_M(K) \leq 27$) of each of these knots is determined.

To conclude this discussion, we look ahead to tile numbers for 7-mosaics. For prime knots with mosaic number 4 or 5, the tile number was completely determined by the mosaic number. For prime knots with mosaic number 6, Theorem \ref{tile-numbers} provides four possibilities for the tile number. For prime knots with mosaic number 7, there are many more possibilities.

\begin{conjecture} If we have a space-efficient 7-mosaic of a prime knot $K$ for which either every column or every row is occupied, then the only possible values for the tile number of the mosaic are 27, 29, 31, 32, 34, 36, 37, 39, 40, and 41.
\end{conjecture}

\bibliographystyle{amsplain}
\bibliography{bibliography}

\providecommand{\bysame}{\leavevmode\hbox to3em{\hrulefill}\thinspace}
\providecommand{\MR}{\relax\ifhmode\unskip\space\fi MR }
\providecommand{\MRhref}[2]{%
  \href{http://www.ams.org/mathscinet-getitem?mr=#1}{#2}
}
\providecommand{\href}[2]{#2}
\begin{thebibliography}{1}

\bibitem{Heap2}
Aaron Heap and Doug Knowles, \emph{Space-efficient knot mosaics for prime knots
  with mosaic number 6}, Preprint:.

\bibitem{Kuriya}
Takahito Kuriya and Omar Shehab, \emph{The {L}omonaco-{K}auffman conjecture},
  J. Knot Theory Ramifications \textbf{23} (2014), no.~1, 20 pp.

\bibitem{Lee}
Hwa~Jeong Lee, Kyungpyo Hong, Ho~Lee, and Seungsang Oh, \emph{Mosaic number of
  knots}, J. Knot Theory Ramifications \textbf{23} (2014), no.~13, 8 pp.

\bibitem{Lee2}
Hwa~Jeong Lee, Lewis~D. Ludwig, Joseph~S. Paat, and Amanda Peiffer, \emph{Knot
  mosaic tabulation}, Involve \textbf{11} (2018), no.~1, 13--26.

\bibitem{Lom-Kauff}
Samuel~J. Lomonaco and Louis~H. Kauffman, \emph{Quantum knots and mosaics},
  Quantum Inf. Process. \textbf{7} (2008), no.~2-3, 85--115.

\bibitem{Ludwig}
Lewis~D. Ludwig, Erica~L. Evans, and Joseph~S. Paat, \emph{An infinite family
  of knots whose mosaic number is realized in non-reduced projections}, J. Knot
  Theory Ramifications \textbf{22} (2013), no.~7, 11 pp.

\bibitem{Thistle}
Morwen Thistlethwaite and Jim Hoste, \emph{Knot{S}cape}, \\ \href{
  http://www.math.utk.edu/~morwen/knotscape.html}{
  http://www.math.utk.edu/\textasciitilde morwen/knotscape.html}.

\end{thebibliography}
\addcontentsline{toc}{section}{\refname}

\end{document}